\documentclass{amsart}
\usepackage[utf8]{inputenc}
\usepackage[english]{babel}
\usepackage{amsmath}
\usepackage{amsfonts}
\usepackage{amssymb}
\usepackage{amsthm}
\usepackage{listings}
\usepackage{hyperref}
\usepackage{mathtools}
\usepackage[all]{xy}
\usepackage{youngtab}
\usepackage{array}
\usepackage{csquotes}
\usepackage{tikz-cd}
\usepackage{stmaryrd}
\usepackage{longtable}
\usepackage{booktabs}
\usepackage{multirow}
\usetikzlibrary{calc}
\usepackage[backrefs,lite]{amsrefs}

\DeclareMathOperator{\im}{im}
\DeclareMathOperator{\Cl}{Cl}

\DeclareMathOperator{\Pic}{Pic}

\DeclareMathOperator{\Aut}{Aut}
\DeclareMathOperator{\lcm}{lcm}

\DeclareMathOperator{\End}{End}
\DeclareMathOperator{\HNF}{HNF}
\DeclareMathOperator{\Norm}{Norm}

\newcommand{\Pp}{\mathbb{P}}

\newcommand{\Z}{\mathbb{Z}}
\newcommand{\Q}{\mathbb{Q}}

\newcommand{\K}{\mathbb{K}}

\newtheorem{lem}{Lemma}[section]
\newtheorem{proposition}[lem]{Proposition}
\newtheorem{corollary}[lem]{Corollary}
\newtheorem{theorem}[lem]{Theorem}
\newtheorem{classification}[lem]{Classification}

\theoremstyle{definition}
\newtheorem{definition}[lem]{Definition}

\newtheorem{construction}[lem]{Construction}
\newtheorem{example}[lem]{Example}

\newtheorem{remark}[lem]{Remark}

\newtheorem{notation}[lem]{Notation}
\newtheorem{procedure}[lem]{Procedure}
\newtheorem{algorithm}[lem]{Algorithm}
\title[A classification algorithm for reflexive simplices]{A classification algorithm\\ for reflexive simplices}
\author{Marco Ghirlanda}
\email{marco.ghirlanda@uni-tuebingen.de}
\address{Mathematisches Institut, Auf d. Morgenstelle 10, 72076 Tübingen}
\date{}
\begin{document}

\begin{abstract}
    We present a general classification algorithm for reflexive simplices,
    which allows us to determine all reflexive simplices in dimensions
    five and six.
    In terms of algebraic geometry this means that we classify the
    Gorenstein fake weighted projective spaces in dimensions five and six.
    As a byproduct of our methods, we obtain explicit formulae for the
    Picard group and the Gorenstein index of any fake weighted
    projective space.
\end{abstract}

\maketitle

\section{Introduction}

A \emph{reflexive polytope} is a lattice polytope having the origin
in its interior and whose dual is also a lattice polytope.
This concept showed up in~\cite{B}, where among other things,
Batyrev determines all reflexive polygons up to unimodular
equivalence.
Kreuzer and Skarke extended this classification in dimensions three
and four, see~\cites{KS3,KS4}. 
A complete classification of all reflexive polygons in dimension 5
or higher seems to be out of reach.
However, Sch\"oller and Skarke could at least compute the 
\emph{weight systems} of the reflexive polytopes in dimension
five, see~\cite{SS}.

In the present note, we consider the particular case of
\emph{reflexive simplices}.
The aforementioned classifications comprise in particular the
5 reflexive triangles, the 48 reflexive 3-simplices and the
1561 reflexive 4-simplices.
The main result of this article is the general classification
algorithm~\ref{algo:classification} for reflexive simplices.
It allows us in particular to obtain the following.

\begin{classification}\label{class:reflexive}
Up to unimodular equivalence, there are $220 \, 794$ reflexive simplices
in dimension five and $309 \, 019 \, 970$ reflexive simplices in dimension six.
The complete data are available at \cite{zenodo:17296449}.
\end{classification}

On a midrange computer, with 16 threads, the algorithm terminates
in less than one minute for dimension five, and in approximately
20 days for dimension six.

We succeed by working mostly in terms of \emph{degree matrices}, a Gale dual encoding of
lattice simplices by means of a matrix with columns in an abelian group~$\Z \times \Gamma$ with~$\Gamma$ finite; see Section~\ref{sec:fwps} for the precise formulation.
This framework has two main advantages. First, unimodular equivalences of the lattice simplices correspond to automorphisms of
$\Z \times \Gamma$. Theorem~\ref{isom theorem} provides explicit generators of the automorphism
group of $\Z \times \Gamma$, which allows early pruning of the possible degree matrices. Second, Proposition \ref{prop:gor_fwps} turns the reflexivity condition for the simplex into explicit conditions on the degree matrix: one condition involves only the free row, and the others one torsion row at a time. Hence, we find all reflexive degree matrices by computing the admissible free rows, enumerating the torsion rows compatible with each free row, and suitably combining them.

Our results directly apply to toric geometry. Recall that the reflexive simplices are exactly the Fano polytopes
of the Gorenstein \emph{fake weighted projective spaces (fwps)}, that means, the $\Q$-factorial toric
varieties of Picard number one whose anticanonical divisor is ample and Cartier.
So, in the setting of toric varieties, Classification \ref{class:reflexive} tells us the
following.

\begin{corollary}\label{corclass}
Up to isomorphism, there are $220 \, 794$ Gorenstein fwps of dimension five
and $309 \, 019 \, 970$ Gorenstein fwps of dimension six.
\end{corollary}

Moreover, the simplices of the classification lists directly yield the defining fan of
the corresponding fwps: the maximal cones are the cones over the facets of the simplex.
Beyond that, we present explicit formulae for the Picard group and
the Gorenstein index of a fwps in terms of its degree matrix, see Theorem \ref{thm:pic_fwps}.

\tableofcontents

\section{Automorphisms of finitely generated abelian groups}\label{sec:groups}

The main result of this section explicitly provides
generators of the automorphism
group of a finitely generated abelian group $G$. A first description of $\Aut(G)$ was given in \cite{R} by Ranum. In \cite{HR}, Hillar and Rhea characterized $\Aut(G)$ when $G$ is a finite abelian group, and provided a formula for its cardinality. Both articles used the primary decomposition of the group $G$. We will instead present any finitely generated abelian group $G$ in 
\emph{invariant factor form}, that means that 
\[
G
\ = \
\Z^k \oplus \Z/\mu_1\Z \oplus \ldots \oplus \Z/\mu_r\Z,
\qquad
\mu_r \mid \mu_{r-1}, \ldots, \mu_2 \mid \mu_1.
\]
This presentation involves in general fewer cyclic factors. This in turn has significant impact on the performance of the algorithms presented later.  

Write $(w,\eta)$
for the elements of $G$, where $w \in \Z^k$ and
$\eta\in \Z/\mu_1\Z \oplus \ldots \oplus \Z/\mu_r\Z$.
Observe that we obtain well-defined automorphisms of $G$,
each modifying the $i$-th coordinate of either $w$ or $\eta$, by
\[
\begin{array}{lcll}  
\psi_i(w,\eta) & := & (w_1,\ldots,-w_i,\ldots,w_k,\eta), & \quad 1 \le i \le k,
\\[3pt]
\psi_{i,u}(w,\eta) & := & (w,\eta_1, \ldots, u\eta_i, \ldots, \eta_r), & \quad 1 \le i \le  r, \ u \in (\Z/\mu_i\Z)^*, 
\\[3pt]
\alpha_{i,j}(w,\eta) & := & (w_1,\ldots, w_j + w_i, \ldots, w_k,\eta), & \quad 1 \le i,j \le k, \ i \ne j,
\\[3pt]
\beta_{i,j}(w,\eta) & := & (w,\eta_1, \ldots, w_j + \eta_i, \ldots, \eta_r), & \quad 1 \le j \le k, 1 \le i \le r,
\\[3pt]
\gamma_{i,j}(w,\eta) & := & (w,\eta_1, \ldots, \eta_j + \eta_i, \ldots, \eta_r), & \quad 1 \le j < i \le r,
\\[3pt]
\delta_{i,j}(w,\eta) & := & (w,\eta_1, \ldots, \frac{\mu_i}{\mu_j}\eta_j + \eta_i, \ldots, \eta_r), & \quad 1 \le i < j \le r.
\end{array}
\] 

\begin{theorem}\label{isom theorem}
Let $G = \Z^k \oplus \Z/\mu_1\Z \oplus \ldots \oplus \Z/\mu_r\Z$ be in invariant factor form.
Then $\Aut(G)$ is generated by the automorphisms
$\psi_i, \psi_{i,u}, \alpha_{i,j}, \beta_{i,j}, \gamma_{i,j}$ and $\delta_{i,j}$.
\end{theorem}

The proof of Theorem \ref{isom theorem} is given at the end of this section.
As a first preparatory step, we develop a matrix calculus for the $\Z$-algebra of endomorphisms
of an abelian group in invariant factor form.

\begin{definition}
\label{def:gmatrix}
Let $G = \Z^k \oplus \Z/\mu_1\Z \oplus \ldots \oplus \Z/\mu_r\Z$ be in invariant factor form.
By a \emph{$G$-matrix} we mean a square matrix of the shape
\[
A \ = \ \left(a_{ij}\right) \ = \ \begin{bmatrix} A_1 & 0 \\ \ast & A_2 \end{bmatrix},
\]
where the columns of $A$ are elements of $G$, the block $A_1$ is a $k \times k$ square matrix,
the block $A_2$ is an $r \times r$ matrix and the entries of $A_2$ satisfy
\[
\mu_i a_{k+i,k+j} \ = \ 0 \ \in \ \Z/\mu_j\Z, \quad 1\leq i<j\leq r.
\]
\end{definition}

\begin{remark}
Let $G$ be an abelian group in invariant factor form,
$\omega$ an element of~$G$ and $A,A'$ two $G$-matrices.  
The conditions on $G$-matrices ensure that
the formal matrix-vector product $A \cdot \omega$
is a well defined element of $G$ and the formal
matrix-matrix product $A \cdot A'$ is a $G$-matrix.
\end{remark}

\begin{construction}
\label{constr:end2mat}
Consider $G=\Z^k\oplus \Z/\mu_1\Z\oplus\dots\oplus\Z/\mu_r\Z$ in invariant factor form
and denote by $e_i\in G$ the element with $i$-th component $1$ and all others $0$.
With any endomorphism $\varphi \colon G \to G$ we associate the $(k+r) \times (k+r)$
square matrix
\[
A(\varphi)
\ := \ 
\left( \varphi_{ij} \right)
\ := \
[\varphi(e_1), \ldots, \varphi(e_{k+r})].
\]
\end{construction}

\begin{lem} \label{endomorphism lemma}
Let $G=\Z^k\oplus \Z/\mu_1\Z\oplus\dots\oplus\Z/\mu_r\Z$ be in invariant factor form,
and let $\varphi\in\End(G)$.
Then the matrix $A(\varphi)$ from Construction \ref{constr:end2mat} is a $G$-matrix.
\end{lem}

\begin{proof}
    The assertion follows directly from $ 0=\varphi(0)=\varphi(\mu_i e_{k+i})=\mu_i \varphi(e_{k+i}).
    $
\end{proof}

\begin{remark}
For any abelian group $G$ in invariant factor form, the matrices from
Construction~\ref{constr:end2mat} represent the endomorphisms:
for all $\varphi,\varphi_1,\varphi_2 \in \End(G)$ and~$\omega \in G$,
we have
\[
\varphi(\omega) \ = \ A(\varphi) \cdot \omega,
\qquad
A(\varphi_2 \circ \varphi_1) \ = \ A(\varphi_2) \cdot A(\varphi_1).
\]
Conversely, every $G$-matrix defines an endomorphism via
matrix-vector multiplication and these assignments are
inverse to each other.
\end{remark}

\begin{example}\label{ex:rowop}
The matrices associated with 
$\psi_i, \psi_{i,u}, \alpha_{i,j}, \beta_{i,j}, \gamma_{i,j}$
and $\delta_{i,j}$ from Theorem~\ref{isom theorem} are
$G$-matrices.
Multiplication of these $G$-matrices from the left (right) to a given $G$-matrix
yields the elementary row (column) operations preserving the structure of $G$-matrix.
\end{example}

The following statement gives a characterization of generator systems of the group $G$.
It will be essentially used in the proof of Theorem~\ref{isom theorem} and in the later
algorithmic Section~\ref{sec:algos}.

\begin{proposition}\label{generation condition}
    Let $G=\Z^k\oplus \Z/\mu_1\Z\oplus\dots\oplus\Z/\mu_r\Z$ be in invariant factor form, and let $\omega_i=(w_i,\eta_i)\in G$ for $i=1,\dots,n$. For $j=0,\dots,r$ set
    $$
    Q_j\coloneqq \begin{bmatrix}
        w_1 &\dots &w_n\\
        \eta_{11} & \dots & \eta_{n1}\\
        \vdots & & \vdots\\
        \eta_{1j} & \dots & \eta_{nj}
    \end{bmatrix}.
    $$
    Then $G=\langle \omega_1,\dots,\omega_n\rangle$ if and only if the maximal minors of $Q_0$ generate $\Z$ and the maximal minors of $Q_j$ generate $\Z/\mu_j\Z$ for any $j=1,\dots,r$. 
\end{proposition}

\begin{proof}
    Let $G=\langle \omega_1,\dots,\omega_n\rangle$. By classical determinantal divisors theory, the maximal minors of $Q_0$ generate $\Z$. Fix $j=1,\dots,r$, and consider the prime factorization~$\mu_j=p_1^{m_1}\cdots p_t^{m_t}$. Now, fix $i=1,\dots,t$, and consider the projection onto the first $k+j$ coordinates followed by the coordinate-wise projection onto $\Z/p_i\Z$:
    $$
    \pi\colon G\rightarrow(\Z/p_i\Z)^{k+i}.
    $$
    Then $\pi(\omega_1),\dots,\pi(\omega_n)$ forms a basis of the vector space $(\Z/p_i\Z)^{k+j}$, encoded by the matrix $Q_j$ modulo $p_i$. In particular, there exists a maximal minor of $Q_j$ generating~$\Z/p_1\Z$, and hence $\Z/p_1^{m_1}\Z$. The assertion follows since, by the Chinese remainder theorem, we have an isomorphism
    $$
    \Z/\mu_j\Z\rightarrow\Z/{p_1^{m_1}}\Z\oplus\dots\oplus\Z/p_t^{m_t}\Z, \quad     \eta\mapsto(\eta\mod p_1^{m_1},\dots,\eta\mod p_t^{m_t}).
    $$
    
    Now assume that the maximal minors of $Q_0$ generate $\Z$ and the maximal minors of $Q_j$ generate $\Z/\mu_j\Z$ for any $j=1,\dots,r$.
    Let $\eta'_{ij}\in\Z_{\geq 0}$ be the representing integer of  $\eta_{ij}$ between $0$ and $\mu_j-1$, and consider the integer matrix
    $$
    Q'_j\coloneqq \begin{bmatrix}
        w_1 &\dots &w_n\\
        \eta'_{11} & \dots & \eta'_{n1}\\
        \vdots & & \vdots\\
        \eta'_{1j} & \dots & \eta'_{nj}
    \end{bmatrix}.
    $$
    Denote by $q_j$ the greatest common divisor of all maximal minors of $Q'_j$. By assumption, we have $q_0=1$ and $\gcd(q_j,\mu_j)=1$.
    With the vectors~$b_i\coloneqq\mu_i e_{k+i}$ we define the $(k+r)\times(n+r)$-matrix
    $$
    M\coloneqq[Q'_r,B],\qquad B\coloneqq[b_1,\dots,b_r].
    $$
    Denote by $\Delta$ the greatest common divisor of all maximal minors of $M$. Notice that~$\Delta$ divides $\mu_r\cdots\mu_{i+1}q_{i}$ for all $i=0,\dots,r$.
    In particular, $\Delta$ divides \begin{align*}
        &\gcd(q_r,\mu_r q_{r-1},\mu_r\mu_{r-1}q_{r-2},\dots,\mu_r\cdots\mu_1q_0)\\
        =&\gcd(q_r,\mu_r\gcd(q_{r-1},\mu_{r-1}q_{r-2},\dots,\mu_{r-1}\cdots\mu_1q_0)).
    \end{align*}
    Since $\gcd(q_r,\mu_r)=1$, we have
    \begin{align*}
    &\gcd(q_r,\mu_r\gcd(q_{r-1},\mu_{r-1}q_{r-2},\dots,\mu_{r-1}\cdots\mu_1q_0))\\
    =&\gcd(q_r,q_{r-1},\mu_{r-1}q_{r-2},\dots,\mu_{r-1}\cdots\mu_1q_0).
    \end{align*}
    Iterating the argument for $\mu_{r-1},\dots,\mu_1$, we see that $\Delta$ divides~$\gcd(q_r,\dots,q_0)$. Since~$q_0=1$, it follows that $\Delta=1$. By classical determinantal divisors theory, it follows that
    $\Z^{k+r}=\langle M_{*,1},\dots,M_{*,n+r}\rangle$. Thus, the equality~$G=\langle \omega_1,\dots,\omega_n\rangle$ follows from
    $$
    \dfrac{G}{\langle \omega_1,\dots,\omega_r\rangle}\cong\dfrac{\Z^{k+r}}{\langle M_{*,1},\dots,M_{*,n+r}\rangle}.
    $$
    
\end{proof}

\begin{lem}\label{gcd lemma}
    Let $a,b,c\in\Z$ such that $\gcd(a,b,c)=1$. Then there exists $k\in\Z$ such that $\gcd(a+kb,c)=1$.
\end{lem}

\begin{proof}
    Let $p_1,\dots,p_t$ be the prime factors of $c$. For every $i=1,\dots,t$, we find an integer~$k_i\in\Z$ such that $p_i$ does not divide~$a+k_ib$.
    If $p_i$ does not divide $a$, we set~$k_i\coloneqq0$.
    If $p_i$ divides $a$, then it does not divide $a+b$, and we set $k_i\coloneqq 1$.
    By the Chinese remainder theorem, the system
    $$
    \left\{
    \begin{aligned}
    k\equiv k_i \mod p_i, \quad i=1,\dots,t\
    \end{aligned}
    \right.
    $$
    admits a solution $k\in\Z$. We conclude that $\gcd(a+kb,c)=1$ since, for any prime factor $p_i$ of $c$, we have
    $$a+kb\equiv a+k_ib\not\equiv0\mod p_i.$$
\end{proof}

The idea for proving Theorem \ref{isom theorem} is simply to turn the matrix $A(\varphi)$ associated with $\varphi\in\Aut(G)$ into the identity matrix by suitably applying the elementary matrices from Example \ref{ex:rowop}.

\begin{proof}[Proof of Theorem \ref{isom theorem}]
    Let $\varphi\in\Aut(G)$, and let $A(\varphi)$ be the matrix from Construction \ref{constr:end2mat}. By Lemma \ref{endomorphism lemma}, we have
    $$
    A(\varphi)=\begin{bmatrix}
        A_1 & 0 \\ \ast & A_2
    \end{bmatrix},\\[3pt]
    $$
    where $A_1$ and $A_2$ are from Definition \ref{def:gmatrix}.
    Since $\varphi$ is an automorphism, $A(\varphi)$ is invertible. In particular, the integer matrix $A_1$ is also invertible. Hence, suitably applying $\alpha_{ij}$ and $\psi_i$ to $\varphi$ turns $A_1$ into the identity matrix.
    Then, suitably applying~$\beta_{ij}$ turns $A(\varphi)$ into the matrix
    $$
    \begin{bmatrix}
        I & 0\\
        0 & A_2
    \end{bmatrix}.
    $$
    By Lemma~\ref{endomorphism lemma}, there exist $x_i\in\Z/\mu_i\Z$ for $i=1,\dots,r$ such that the last column of~$A_2$ is given as
    $$\left(\frac{\mu_1}{\mu_r}x_1,\dots,\frac{\mu_{r-1}}{\mu_r}x_{r-1},x_r\right).$$
    Let $x_i'\in\Z_{\geq 0}$ be the representing integer of $x_i$ between $0$ and $\mu_i-1$.
    By Proposition~\ref{generation condition}, the determinant of $A_2$ generates $\Z/\mu_r\Z$. In particular, we have
    $$\gcd\left(\frac{\mu_1}{\mu_r}x_1',\dots,\frac{\mu_{r-1}}{\mu_r}x_{r-1}',x_r',\mu_r\right)=1.$$ By Lemma \ref{gcd lemma}, there exists $k\in\Z$ such that
    $$x_r'+k\gcd\left(\frac{\mu_1}{\mu_r}x_1',\dots,\frac{\mu_{r-1}}{\mu_r}x_{r-1}'\right)=u\in(\Z/\mu_r\Z)^\times.
    $$
    Hence, we can use suitable $\gamma_{i,r}$ and then normalize the last coordinate with $\psi_{r,u^{-1}}$ to turn the last column of $A_2$ into
    $$
    \left(\frac{\mu_1}{\mu_r}x_1,\dots,\frac{\mu_{r-1}}{\mu_r}x_{r-1},1\right).
    $$
    Then, suitably applying $\delta_{i,r}$ turns the last column of~$A_2$ into $(0,\dots,0,1)$.
    Iterating the procedure for the other columns turns $A_2$ into a lower diagonal matrix with diagonal entries equal to one.
    Finally, we use suitable $\gamma_{i,j}$ to turn $A_2$ into the identity matrix.
\end{proof}

\section{Picard group and Gorenstein index of fwps}\label{sec:fwps}

In this section we provide explicit formulae for the Picard group and the Gorenstein index of a \emph{fake weighted projective space (fwps)}, that means a $\Q$-factorial toric Fano variety of Picard number one. Every fwps of dimension $n$ is encoded by an~$n\times(n+1)$ \emph{generator matrix}, that means a matrix
$$
P=\begin{bmatrix}
  v_0 & \dots & v_n  
\end{bmatrix},
$$
the columns of which are pairwise distinct primitive vectors generating $\Q^n$ as a convex cone. The fwps associated with $P$ is the toric Fano variety $Z(P)$ whose Fano polytope has vertices $v_0,\dots,v_n$. The class group of $Z$ is isomorphic to~$\Z^{n+1}/\im(P^*)$, where $P^*$ denotes the transpose of $P$. We present $\Cl(Z)$ in invariant factor form
$$\Z\oplus \Z/\mu_1\Z\oplus\dots\oplus\Z/\mu_r\Z.$$
Consider the projection $Q\colon \Z^{n+1}\rightarrow\Cl(Z)$, and let~$\omega_i=(w_i,\eta_i)\coloneqq Q(e_i)$, where $1\leq w_0\leq\dots\leq w_n$ and $\eta_i\in\Z/\mu_1\Z\oplus\dots\oplus\Z/\mu_r\Z$. We view $Q$ as a \emph{degree matrix in}~$\Cl(Z)$, that means
$$
Q=\begin{bmatrix}
    \omega_0 & \dots &\omega_n
\end{bmatrix}
$$
where any $n$ of the $\omega_i$ generate $\Cl(Z)$ as a group.
One can directly gain $Z$ as a quotient of $\K^{n+1}$ by the diagonal action of $H=\K^*\times F$ with weights $\omega_i$, where $F$ is finite and $H$ is the quasitorus with character group $\Cl(Z)$, see \cite{HHHS}*{Sec. 2}.

\begin{example}\label{ex:fwps}
    We consider the fwps $Z=Z(P)$ with generator matrix $P$ and degree matrix $Q$ in $\Cl(Z)=\Z\oplus \Z/4\Z$:
    $$
    P=\begin{bmatrix}
        1 & -2 & 1 & 0\\
        -2 & -2 & 0 & 1\\
        -3 & 2 & 1 & 0
    \end{bmatrix},\qquad 
    Q=\begin{bmatrix}
        1 & 1 & 1 & 4 \\
        \bar{0} & \bar{1} & \bar{2} & \bar{2}
    \end{bmatrix}.
    $$
    The quasitorus associated to the grading is $H=\K^*\times\{\pm 1, \pm \sqrt{-1}\}$, and $Z$ is isomorphic to the quotient of $\K^4$ by the action of $H$:
    $$
    (t,\sqrt{-1})\cdot z\coloneqq[tz_1,\sqrt{-1}tz_2,-tz_3,-t^4z_4].
    $$
\end{example}

\begin{notation}\label{not}
    Let $G=\Z\oplus \Z/\mu_1\Z\oplus\dots\oplus\Z/\mu_r\Z$ be in invariant factor form, and let $\omega_i=(w_i,\eta_i)\in G$ for $i=0,\dots,n$.
    We denote by $\eta_{ij}'\in\Z$ the representing integer of $\eta_{ij}$ between $0$ and $\mu_j-1$, and we set
    \[
    \begin{array}{lcll}  
    L & \coloneqq & \lcm(w_0,\dots,w_n),
    \\[3pt]
    L_{ij} & \coloneqq & \dfrac{L}{w_i}\eta_{ij}', \quad 0 \le i \le n, \ 1\leq j\leq r,
    \\[3pt]
    M_i & \coloneqq & \dfrac{\mu_i}{\gcd\left(\mu_i,L_{0i},\dots,L_{ni}\right)}, \quad 1 \le i \le r,
    \\[9pt]
    M & \coloneqq & \lcm(M_1,\dots,M_r).
    \end{array}
    \] 
\end{notation}

\begin{theorem} \label{thm:pic_fwps}
    Let $Z$ be a fwps with $\Cl(Z)=\Z\oplus \Z/\mu_1\Z\oplus\dots\oplus\Z/\mu_r\Z$ and  degree matrix $Q=\begin{bmatrix}
        \omega_0 & \dots & \omega_n
    \end{bmatrix}$.
    With Notation \ref{not}, we have
    \vspace{5pt}
    $$
    \begin{array}{lcl}
    \Pic(Z) & = &\langle(LM,0)\rangle,\\[8pt]
    \iota(Z) & = &\lcm\left(\dfrac{LM}{\gcd\left(LM,\sum_{i}w_i\right)},\dfrac{\mu_j}{\gcd\left(\mu_j,\sum_i \eta'_{ij}\right)}; \ j=1,\dots,r\right).
    \end{array}
    $$
\end{theorem}

We obtain Theorem \ref{thm:pic_fwps} from the following statement about finitely generated abelian groups of rank one. 

\begin{proposition} \label{pic_fwps}
    Let $G=\Z\oplus \Z/\mu_1\Z\oplus\dots\oplus\Z/\mu_r\Z$ be in invariant factor form, and let $\omega_i=(w_i,\eta_i)\in G$ for $i=0,\dots,n$. Suppose $\omega_0,\dots,\omega_n$ generate~$G$. Then, with Notation \ref{not}, we have
    $$
    \bigcap_{i=0}^n\langle \omega_i\rangle=\langle(LM,0)\rangle.
    $$
\end{proposition}

\begin{proof}
    The inclusion "$\subseteq$" follows from
    $
    LM\omega_i=w_i(LM,0)$ for $i=1,\dots,n.
    $
    
    We show "$\supseteq$". Let $\omega=(w,\zeta)\in G$ such that $\omega=k_i\omega_i$ for integers $k_i\in\Z$, $i=0,\dots,n$. We first prove that $\zeta_1= 0\in\Z/\mu_1\Z$.
    Let $\mu_1$ factor as $p_1^{m_1}\cdots p_s^{m_s}$.
    Consider the projection onto the first two coordinates followed by the coordinate-wise projection onto $\Z/p_i^{m_i}\Z$:
    $$
    \pi\colon  G\rightarrow (\Z/p_i^{m_i}\Z)^2
    $$
    By Nakayama's lemma, $\pi(\omega_j),\pi(\omega_\ell)$ generate~$(\Z/p_i^{m_i}\Z)^2$ for some $1\leq j,\ell\leq n$.
    It follows that
    $$\det\begin{bmatrix}\pi(\omega_j) & \pi(\omega_\ell)
    \end{bmatrix}\in(\Z/p_i^{m_i}\Z)^\times.$$
    From $k_j\omega_j=k_\ell\omega_\ell$ we have
    $$    0=\det\begin{bmatrix}\pi(k_j\omega_j-k_\ell\omega_\ell) & \pi(\omega_\ell)
    \end{bmatrix}=k_j\det\begin{bmatrix}\pi(\omega_j) & \pi(\omega_\ell)
    \end{bmatrix}.
    $$
    Hence, $p_i^{m_i}$ divides $k_j$. In particular
    $\zeta_1=k_j\eta_{j1}\equiv 0 \mod p_i^{m_i}$.
    By the Chinese remainder theorem, we have $\zeta_1=0\in\Z/\mu_1\Z$. Iterating the argument for $\mu_2,\dots,\mu_r$, we prove that $\zeta=0$.
    
    Now, since $w_i$ divides $w$ for all $i$, there exists $k\in\Z$ such that $w=Lk$. In particular, we have
    $$
    kL_{ij}=k_i\eta_{ij}=0\in\Z/\mu_j\Z.
    $$
    It follows that $\mu_j$ divides
    $
    k\gcd\left(L_{0j},\dots,L_{nj}\right).
    $
    Hence, $M_j$ divides $k$ for all $j$. In turn, this implies that $LM$ divides $w$.
\end{proof}

\begin{proof}[Proof of Theorem \ref{thm:pic_fwps}]
    Since $\omega_0,\dots,\omega_n$ generate $\Cl(Z)$, the first assertion follows from Proposition \ref{pic_fwps}.
    Now, let~$k\in\Z_{\geq 0}$ and recall that $-\mathcal{K}_Z=\sum_i\omega_i$. Then $k(-\mathcal{K})$ is Cartier if and only if
    $$
    \begin{cases}
        LM &| \ \ k\sum_i w_i,\\
        \ \ \mu_j &| \ \ k\sum_i\eta'_{ij},\quad j=1,\dots,r.
    \end{cases}
    $$
    
\end{proof}

\begin{example}
    Let $Z$ be the fwps from Example \ref{ex:fwps}. Then $\Pic(Z)=\langle(8,\bar{0})\rangle$ and~$\iota(Z)=8$.
\end{example}

\section{Classification of reflexive simplices}\label{sec:algos}

As a direct consequence of Theorem \ref{thm:pic_fwps}, we have the following characterization of Gorenstein fwps in terms of their degree matrix.

\begin{proposition}\label{prop:gor_fwps}
    Let $Z$ be a fwps with $\Cl(Z)=\Z\oplus \Z/\mu_1\Z\oplus\dots\oplus\Z/\mu_r\Z$ and  degree matrix $Q=\begin{bmatrix}
        \omega_0 & \dots & \omega_n
    \end{bmatrix}$.
    Then $Z$ is Gorenstein if and only if, with Notation \ref{not}, we have
    $$
    \begin{cases}
        L & | \ \ \sum_i w_i,\\[3pt]
        M_j & | \ \ \dfrac{\sum_{i}w_i}{L},\quad j=1,\dots,r,\\[5pt]
        \eta_{nj} & = \ \ -(\eta_{0j}+\dots+\eta_{n-1j})\in\Z/\mu_j\Z,\quad j=1,\dots,r.
    \end{cases}
    $$
\end{proposition}

Notice that the first condition of Proposition \ref{prop:gor_fwps} involves only the torsion-free row of $Q$. In turn, the other two conditions are the same for all torsion rows, and only involve one torsion row at a time. This justifies the following definitions.

\begin{definition}\label{weight_vec}
    We call $w=[w_0,\dots,w_n]\in\Z^{n+1}_{\geq 1}$ a \emph{(Gorenstein) weight vector} if it is the degree matrix of a (Gorenstein) fwps $Z=Z(w)$ with $\Cl(Z)=\Z$. In this case, we write $\Pp(w_0,\dots,w_n)\coloneqq Z(w)$ and call it a \emph{weighted projective space~(wps)}.
\end{definition}

\begin{definition}\label{def: torsion_vec}
    Let $w=[w_0,\dots,w_n]$ be a weight vector and $\mu\in\Z_{\geq 2}$. We call~$\eta=[\eta_0,\dots,\eta_n]\in(\Z/\mu\Z)^{n+1}$ a \emph{(Gorenstein) torsion vector of order $\mu$ for $w$} if the matrix
    $$Q=\begin{bmatrix}
        w_0&\dots & w_n\\
        \eta_0 & \dots & \eta_n
    \end{bmatrix}$$
    is the degree matrix of a (Gorenstein) fwps $Z=Z(Q)$ with $\Cl(Z)=\Z\oplus\Z/\mu\Z$.
    \end{definition}
    
    \begin{definition}
    Let $w=[w_0,\dots,w_n]$ be a weight vector and $\mu\in\Z_{\geq 2}$. We say that two torsion vectors $\eta$ and $\zeta$ of order $\mu$ for $w$ are \emph{equivalent} if there exists an automorphism of~$\Z\oplus\Z/\mu\Z$ sending $(w_i,\eta_i)$ to $(w_i,\zeta_i)$ for $i=0,\dots,n$. We write~$\eta\leq\zeta$ if $[\eta'_0,\dots,\eta'_n]\leq_{lex}[\zeta_0',\dots,\zeta_n']$, where $\eta_i'$ and $\zeta_i'$ are the representative integers for $\eta_i$ and $\zeta_i$ between $0$ and $\mu-1$. We say that $\eta$ is \emph{minimal} if $\eta\leq\zeta$ for all $\zeta$ equivalent to $\eta$.
\end{definition}

\begin{proposition}\label{prop:gor_matrix}
    Let $Z$ be a Gorenstein fwps with $\Cl(Z)=\Z\oplus\Z/\mu_1\Z\oplus\dots\oplus\Z/\mu_r\Z$. Then $Z=Z(Q)$ for a degree matrix $$Q=\begin{bmatrix}
        w_0 & \dots & w_n\\
        \eta_{01} & \dots & \eta_{n1} \\
        \vdots & & \vdots \\
        \eta_{0r} & \dots & \eta_{nr} \\
    \end{bmatrix},$$
    where $w=[w_0,\dots,w_n]$ is a Gorenstein weight vector, $\eta_{\ast i}\coloneqq[\eta_{0i},\dots,\eta_{ni}]$ is a minimal Gorenstein torsion vector of order $\mu_i$ for $w$, and $\eta_{\ast i}< \eta_{\ast j}$ for all $i<j$ such that $\mu_i=\mu_j$.
\end{proposition}

\begin{proof}
    Let $Q$ be any degree matrix for $Z$. By Theorem \ref{thm:pic_fwps}, the first row $w$ of $Q$ is a Gorenstein weight vector, and the $(i+1)$-th row of $Q$ is a Gorenstein torsion vector of order $\mu_i$ for $w$, for $i=1,\dots,r$. By Theorem \ref{isom theorem}, there exists an automorphism of $\Cl(Z)$ that turns all torsion vectors minimal and orders increasingly those with the same torsion order.
\end{proof}

We turn to our classification algorithm \ref{algo:classification}.
First we present the necessary ingredients, which then fit together in the final algorithm.

The following, presented for instance in~\cite{B}*{Sec. 5.4}, gives a bijection between Gorenstein weight vectors and the unit fraction decompositions of one.

\begin{remark}\label{rmk:wps}
    Let $W$ be the set of all Gorenstein weight vectors $(w_0,\dots,w_n)$, and
    \begin{align*}
    &U\coloneqq \left\{(u_0,\dots,u_n)\in\Z_{\geq 1}^{n+1}; \ u_0\geq \dots\geq u_n, \ \frac{1}{u_0}+\dots+\dfrac{1}{u_n}=1\right\}.
    \end{align*}
    Then, with $u\coloneqq\lcm(u_0,\dots,u_n)$ and $S\coloneqq w_0+\dots+w_n$, we have mutually inverse bijections
    $$
    \begin{array}{ccc}
        U & \longleftrightarrow &  W\\[5pt]
        (u_0,\dots,u_n) & \mapsto & \left(\frac{u}{u_0},\dots,\frac{u}{u_n}\right),\\[5pt]
        \left(\frac{S}{w_0},\dots,\frac{S}{w_n}\right) & \mapsfrom & (w_0,\dots,w_n).
    \end{array}
    $$
    There exist several algorithms for computing the set $U$, see for instance \cite{B}*{Algorithm 5.6}.
    Moreover, the lists for $n=1,\dots,8$ are available at \cite{OEIS}.
\end{remark} 

We now determine for which $\mu\in\Z_{\geq 2}$ there exists at least one Gorenstein torsion vector of order $\mu$ for a weight vector $w$.

\begin{remark} \label{rmk: torsion}
    Let $Z$ be a fwps with $\Cl(Z)=\Z\oplus \Z/\mu_1\Z\oplus\dots\oplus\Z/\mu_r\Z$ and  degree matrix $Q=\begin{bmatrix}
        \omega_0 & \dots & \omega_n
    \end{bmatrix}$. By Proposition $\ref{generation condition}$, $\eta_{0j},\dots,\eta_{n_j}$ generate~$\Z/\mu_j\Z$ for~$j=1,\dots,r$. Hence, with Notation \ref{not}, we have
    \begin{enumerate}
        \item $\dfrac{\mu_j}{M_j}=\gcd(\mu_j,L_{0j},\dots,L_{nj}) \ | \ L$;\\
        \item if $Z$ is Gorenstein, then $\mu\ |\ S$.
    \end{enumerate}
\end{remark}

By Remark \ref{rmk: torsion}, all the orders admitting at least one torsion vector for a Gorenstein weight vector $w$ are of the form~$\mu=ab$, where $a$ divides $L$ and $b$ divides~$\frac{S}{L}$. The following lemma provides a criterion to quickly exclude some pairs $(a,b)$, improving significantly the performance of our algorithm.

\begin{lem}\label{lem:pairs}
     Let $Z$ be a fwps with $\Cl(Z)=\Z\oplus \Z/\mu_1\Z\oplus\dots\oplus\Z/\mu_r\Z$ and  degree matrix $Q=\begin{bmatrix}
        \omega_0 & \dots & \omega_n
    \end{bmatrix}$. With Notation \ref{not}, set $\alpha_k\coloneqq\dfrac{LM_k}{\mu_k}$ for $k=1,\dots,r$. Then
    \begin{align*}
    &\dfrac{w_i}{\gcd(\alpha_k,w_i)} \ | \ \eta'_{ik}, \quad i=0,\dots,n,\\[3pt]
    &\gcd\left(\mu_k,\dfrac{w_iw_j}{\gcd(\alpha_k,w_iw_j)} ;\ i,j=0,\dots,n\right)=1.
    \end{align*}
\end{lem}

\begin{proof}
    For the first assertion, observe
    $$
    \dfrac{w_i}{\gcd(\alpha_k,w_i)}=\dfrac{\dfrac{L}{\alpha_k}}{\dfrac{L}{\lcm(\alpha_k,w_i)}}=\dfrac{\gcd(\mu_k,L_{1k},\dots,L_{nk})}{\gcd\left(\mu_k,L_{1k},\dots,L_{nk},\dfrac{L}{w_i}\right)}\ |\ \eta'_{ik}.$$
    In particular, for any $i,j=0,\dots,n$, $k=1,\dots,r$ we have
    $$
    \dfrac{w_iw_j}{\gcd(\alpha_k,w_iw_j)}\ | \ w_i\eta'_{jk}-w_j\eta'_{ik}.
    $$
    The second assertion follows since, by Proposition \ref{generation condition}, we have
    $$\gcd(\mu_k,w_i\eta'_{jk}-w_j\eta'_{ik};\ i,j=0,\dots,n)=1.$$
\end{proof}

\begin{lem}\label{lem:lex_min}
    Let $w=[w_0,\dots,w_n]$ be a weight vector, $\mu\in\Z_{\geq 2}$ and $\eta=[\eta_0,\dots,\eta_n]$ a torsion vector of order $\mu$ for $w$. Write $\eta'_i$ for the representative integer of $\eta_i$ between $0$ and $\mu-1$, and set
    $$    d_i\coloneqq\dfrac{\mu\gcd(\mu,w_0,\dots,w_i)}{\gcd(\mu,w_0,\dots,w_{i-1})}.
    $$
    Then $\eta$ is minimal if and only if 
    $$
    \begin{cases}
        0\leq \eta'_i<d_i, \quad i=0,\dots,n,\\
        \eta\leq u\eta, \quad u\in (\Z/\mu\Z)^\times.
    \end{cases}
    $$
\end{lem}

\begin{proof}
    By Theorem \ref{isom theorem}, a torsion vector $\zeta$ of order $\mu$ for $w$ is equivalent to $\eta$ if and only if $\zeta=u\eta+kw$ for $u\in(\Z/\mu\Z)^\times$, $k\in\Z$.
    By Bézout's identity there exist~$a_0,b_0\in\Z$ such that $$0\leq\eta_0'+a_0\mu+b_0w_0<\gcd(\mu,w_0)=d_0.$$
    Hence, the minimal torsion vector equivalent to $\eta$ must also satisfy this condition.
    Furthermore, adding $kw$ to $\eta$ preserves the inequalities above if and only if $\mu$ divides~$kw_0$. By Bézout's identity, there exist $a_1,b_1\in\Z$ such that 
    \begin{align*}
    & \leq\eta_0'+a_1\mu+b_1w_0<\gcd(\mu,w_0)=d_0, \\    & 0\leq\eta_1'+a_1\mu+b_1w_1<\dfrac{\mu\gcd(\mu,w_0,w_1)}{\gcd(\mu,w_0)}=d_1.
    \end{align*}
    Iterating the argument for $i=2,\dots,n$, we prove that the conditions of the assertion are necessary for $\eta$ to be minimal. 
    Since only one torsion vector equivalent to $\eta$ satisfies the conditions, they are also sufficient.
\end{proof}

The following statement provides an efficient method for computing all minimal Gorenstein torsion vectors for a weight vector $w$.

\begin{procedure}\label{proc:torsion_vectors}
    Let $w=[w_0,\dots,w_n]$ be a Gorenstein weight vector, and consider any pair $(a,b)\in\Z_{\geq 1}^2$ such that
    $$a \ | \ L, \quad b \ | \ \frac{S}{L}, \quad \gcd\left(ab,\frac{w_iw_j}{\gcd(\frac{L}{b},w_iw_j)} ;\ i,j=0,\dots,n\right)=1.
    $$
    For $i=0,\dots,n$ set
    $$
    c_i\coloneqq\dfrac{w_i}{\gcd\left(\frac{L}{b},w_i\right)}, \quad d_i\coloneqq\dfrac{ab\gcd(ab,w_0,\dots,w_i)}{\gcd(ab,w_0,\dots,w_{i-1})},\quad K_i\coloneqq\left\{0,\dots,\left\lceil\frac{d_i}{c_i}\right\rceil-1\right\}.\\[3pt]
    $$
    Let $K$ be the set of all $[k_0,\dots,k_n]\in\prod_i K_i$ such that
    $$
    \begin{cases}
        c_nk_n\equiv-(c_0k_0+\dots+c_{n-1}k_{n-1})\mod ab, \\[5pt]
        \gcd\left(ab,c_ik_iw_j-c_jk_jw_i;\ i,j=0,\dots,\hat{\ell},\dots,n\right)=1, \ \ell=0,\dots,n,\\[5pt]
        \gcd\left(ab,\dfrac{Lc_0k_0}{w_0},\dots,\dfrac{Lc_nk_n}{w_n}\right)=a.
    \end{cases}
    $$
    Then $[\eta_0,\dots,\eta_n]$ is a minimal Gorenstein torsion vector of order $ab$ for $w$ if and only if $\eta\leq u\eta$ for all $u\in(\Z/ab\Z)^\times$ and there exists $[k_0,\dots,k_n]\in K$ such that $c_ik_i$ is a representative integer of $\eta_i$ for $i=0,\dots,n$.
\end{procedure}

\begin{proof}
    The assertion follows from Proposition \ref{prop:gor_fwps}, Lemma \ref{lem:pairs} and  Lemma \ref{lem:lex_min}.
\end{proof}

We now determine when the degree matrix of a fwps can be extended by a torsion row.

\begin{proposition}\label{prop:gluing}
    Let $Z$ be a fwps with $\Cl(Z)=\Z\oplus\Z/\mu_1\Z\oplus\dots\oplus\Z/\mu_r\Z$ and degree matrix $$Q=\begin{bmatrix}
        \omega_0 & \dots & \omega_n
    \end{bmatrix}=\begin{bmatrix}
        w_0 & \dots & w_n\\
        \eta_{0} & \dots & \eta_{n}
    \end{bmatrix}.$$
    Consider a torsion vector $\zeta=[\zeta_0,\dots,\zeta_n]$ of order $\mu\in\Z_{\geq 2}$ for $w=[w_0,\dots,w_n]$, and set
    $$
    Q_\zeta\coloneqq\begin{bmatrix}
        \omega_0 & \dots & \omega_n \\
        \zeta_0 & \dots & \zeta_n
    \end{bmatrix},\quad Q_{\zeta,i}\coloneqq\begin{bmatrix}
        \omega_0 & \dots & \hat{\omega_i} & \dots & \omega_n \\
        \zeta_0 & \dots & \hat{\zeta_i} & \dots & \zeta_n
    \end{bmatrix}.\\[3pt]
    $$
    Then, $Q_\zeta$ is the degree matrix of a fwps $Z_\zeta=Z(Q_\zeta)$ if and only if $\mu$ divides $\mu_r$ and the maximal minors of $Q_{\zeta,i}$ generate $\Z/\mu\Z$ for $i=0,\dots,n$.
\end{proposition}

\begin{proof}
    If $Q_\zeta$ is a degree matrix, then $\mu$ divides $\mu_r$ because $\Cl(Z_\zeta)$ is in invariant factor form. Since any $n$ of the $(\omega_i,\zeta_i)$ generate $\Cl(Z_\zeta)$, the maximal minors of~$Q_{\zeta,i}$ generate $\Z/\mu\Z$ for $i=0,\dots,n$ by Proposition \ref{generation condition}. Viceversa, if $\mu$ divides $\mu_r$,  then~$\Cl(Z)\oplus \Z/\mu\Z$ is in invariant factor form. Since $Q$ is a degree matrix, any $n$ of the $\omega_i$ generate $\Cl(Z)$. Hence, by Proposition \ref{generation condition}, if the maximal minors of $Q_{\zeta,i}$ generate $\Z/\mu\Z$ for $i=0,\dots,n$ then any $n$ of the $(\omega_i,\zeta_i)$ generate $\Cl(Z)\oplus \Z/\mu\Z$.
\end{proof}

\begin{corollary}\label{cor:incomp_row}
    Let $Z$ be a fwps with $\Cl(Z)=\Z\oplus\Z/\mu_1\Z\oplus\dots\oplus\Z/\mu_r\Z$ and degree matrix $$Q=\begin{bmatrix}
        \omega_0 & \dots & \omega_n
    \end{bmatrix}=\begin{bmatrix}
        w_0 & \dots & w_n\\
        \eta_{0} & \dots & \eta_{n}
    \end{bmatrix},\\[4pt]$$ and let $\zeta,\theta$ be two torsion vectors for $w=[w_0,\dots,w_n]$. By Proposition \ref{generation condition} and Proposition \ref{prop:gluing}, if $Q_\zeta$ is a degree matrix and $Q_\theta$ is not a degree matrix, then $(Q_\zeta)_\theta$ is also not a degree matrix.
\end{corollary}

\begin{procedure}\label{proc:gluing}
    Let $w=[w_0,\dots,w_n]$ be a Gorenstein weight vector, and let~$H$ be the set of all pairs $(\mu,\eta)$, where $\mu\in\Z_{\geq 2}$ and $\eta\in(\Z/\mu\Z)^{n+1}$ is a minimal Gorenstein torsion vector of order $\mu$ for $w$. We build the degree matrices of all Gorenstein fwps with weight vector $w$ by progressively adding torsion factors to the class group.
    Clearly, $\Pp(w_0,\dots,w_n)$ is the only wps with weight vector $w$. By Proposition \ref{prop:gor_matrix}, all Gorenstein fwps with $\Cl(Z)=\Z\oplus \Z/\mu_1\Z$ admit a degree matrix
    $$
    Q_{\eta_1}\coloneqq\begin{bmatrix}
        w_0 & \dots & w_n\\
        \eta_{01} & \dots & \eta_{n1}
    \end{bmatrix},
    $$
    where $(\mu_1,[\eta_{01},\dots,\eta_{n1}])\in H$.
    Now, for each $(\mu_1,[\eta_{01},\dots,\eta_{n1}])\in H$, let $H_{\eta_1}$ be the set of all $(\mu_2,[\eta_{02},\dots,\eta_{n2}])$ in $H$ such that $[\eta_{01},\dots,\eta_{n1}]<[\eta_{02},\dots,\eta_{n2}]$ if $\mu_1=\mu_2$ and the matrix
    $$
    Q_{\eta_1,\eta_2}\coloneqq\begin{bmatrix}
        w_0 & \dots & w_n\\
        \eta_{01} & \dots & \eta_{n1}\\
        \eta_{02} & \dots & \eta_{n2}
    \end{bmatrix}
    $$
    is the degree matrix of a Gorenstein fwps. This is the case if and only if $Q_{\eta_1,\eta_2}$ satisfies the conditions of Proposition \ref{prop:gluing}. By Proposition \ref{prop:gor_matrix}, all Gorenstein fwps with $\Cl(Z)=\Z\oplus\Z/\mu_1\Z\oplus\Z/\mu_2\Z$ admit such a degree matrix. Now, for each such matrix $Q_{\eta_1,\eta_2}$, let $H_{\eta_1,\eta_2}$ be the set of all $(\mu_3,[\eta_{03},\dots,\eta_{n3}])\in H_{\eta_1}$ such that~$[\eta_{02},\dots,\eta_{n2}]<[\eta_{03},\dots,\eta_{n3}]$ if $\mu_2=\mu_3$ and the matrix
    $$
    Q_{\eta_1,\eta_2,\eta_3}\coloneqq\begin{bmatrix}
        w_0 & \dots & w_n\\
        \eta_{01} & \dots & \eta_{n1}\\
        \eta_{02} & \dots & \eta_{n2}\\
        \eta_{03} & \dots & \eta_{n3}
    \end{bmatrix}
    $$
    is the degree matrix of a Gorenstein fwps. This is the case if and only if $Q_{\eta_1,\eta_2,\eta_3}$ satisfies the conditions of Proposition \ref{prop:gluing}. By Proposition \ref{prop:gor_matrix} and Corollary \ref{cor:incomp_row}, all Gorenstein fwps with $\Cl(Z)=\Z\oplus\Z/\mu_1\Z\oplus\Z/\mu_2\Z\oplus \Z/\mu_3\Z$ admit such a degree matrix. We iterate this process until $H_{\eta_1,\dots,\eta_r}$ is the empty set. By Proposition \ref{generation condition}, the procedure terminates in at most $n$ steps. 
\end{procedure}

In order to detect when two generator matrices yield isomorphic fwps, we introduce a normal form for a generator matrix $P$. It shares similarities with the PALP normal form presented in~\cite{KS} and the normal form presented in~\cite{Bauerle2}. We denote by $\HNF(P)$ the Hermite normal form of the matrix $P$.

\begin{definition}
    Let $Z$ be a fwps with generator matrix $P$ and degree matrix $Q$:
    $$
    P=\begin{bmatrix}
        v_0 & \dots & v_n
    \end{bmatrix},\quad Q=\begin{bmatrix}
        w_0 & \dots & w_n\\
        \eta_0 & \dots & \eta_n
    \end{bmatrix},
    $$
    where $v_i\in\Z^{n}$, $w_i\in\Z_{\geq 1}$ and $\eta_{i}\in\Z/\mu_1\Z\oplus\dots\oplus\Z/\mu_r\Z$ for $i=0,\dots,n$. Let $S_Z$ be the set of all permutations $\sigma\in S_{n+1}$ such that $\sigma(i)<\sigma(j)$ if $w_i<w_j$. We define the \emph{normal form} of $P$ as 
    $$
    \Norm(P)\coloneqq\min(\HNF(P_\sigma);\ \sigma\in S_Z),
    $$
    where $P_\sigma=[v_{\sigma(0)},\dots,v_{\sigma(n)}]$, and the minimum is taken with respect to the lexicographic order row by row.
\end{definition}

\begin{remark}\label{rmk:Norm}
    Two fwps $Z=Z(P)$ and $Z'=Z(P')$ are isomorphic if and only if~$\Norm(P)=\Norm(P')$.
\end{remark}

\begin{procedure}\label{proc:filter}
\emph{Input:} a positive integer $n$; the collection $Q(n)$ of degree matrices produced by Procedure~\ref{proc:gluing}.\\ \emph{Procedure:}
\begin{enumerate}
  \item Initialize $P(n)\coloneqq\varnothing$.
  \item For each $D\in Q(n)$:
    \begin{enumerate}
      \item compute an associated generator matrix $P$;
      \item compute its normal form $\Norm(P)$;
      \item if $\Norm(P)\notin P(n)$, set $P(n)\coloneqq P(n)\cup\{\Norm(P)\}$.
    \end{enumerate}
  \item Return $P(n)$.
\end{enumerate}
\emph{Output}: a set $P(n)$ of generator matrices containing exactly one representative from each isomorphism class of $n$-dimensional Gorenstein fwps.
\end{procedure}

\begin{algorithm}[Classification of the Gorenstein fwps]\label{algo:classification}
\emph{Input:} a positive integer~$n$.
\emph{Algorithm:}
\begin{enumerate}
  \item\label{step 1} Compute or load the list of all Gorenstein weight vectors $w\coloneqq [w_0,\dots,w_n]$ according to Remark \ref{rmk:wps};
  \item\label{step 2} For each Gorenstein weight vector $w$ and each $\mu\in\Z_{\ge 2}$, compute all minimal Gorenstein torsion vectors of order $\mu$ for $w$ using Procedure \ref{proc:torsion_vectors};
  \item\label{step 3} For each Gorenstein weight vector $w$, compute all degree matrices obtained by suitably combining $w$ with the torsion vectors from~(\ref{step 2}) using Procedure~\ref{proc:gluing};
  \item\label{step 4} Compute the set of associated generator matrices to the degree matrices from~(\ref{step 3}), then select a subset that contains exactly one generator matrix from each isomorphism class of fwps, using Procedure \ref{proc:filter}.
\end{enumerate}
\emph{Output:} a list of representatives for the isomorphism classes of $n$-dimensional Gorenstein fwps, complete according to Proposition \ref{prop:gor_matrix}.\\
\end{algorithm}

\end{document}